\documentclass[12pt]{amsart}

\usepackage{latexsym}
\usepackage{amssymb}
\usepackage{amsmath}
\usepackage{amsfonts}

\usepackage{setspace,  
hyperref, epigraph} 

\usepackage{breakurl}   

\usepackage{enumitem} 

\usepackage{tikz}
\usetikzlibrary{decorations.pathreplacing}

\newtheorem{theorem}{Theorem}[section]
\newtheorem{lemma}[theorem]{Lemma}         
\newtheorem{proposition}[theorem]{Proposition}    
\newtheorem{corollary}[theorem]{Corollary}         

\theoremstyle{definition}
\newtheorem{definition}[theorem]{Definition}

\theoremstyle{remark}
\newtheorem{remark}[theorem]{Remark}

\newcommand{\bbR}{\mathbb{R}}

\newcommand{\bbH}{\mathbb{H}}

\newcommand{\bbP}{\mathbb{P}}
\newcommand{\bbZ}{\mathbb{Z}}
\newcommand{\bbN}{\mathbb{N}}

\newcommand{\bbQ}{\mathbb{Q}}

\newcommand{\cM}{\mathfrak{M}}

\newcommand{\cP}{\mathcal{P}}

\newcommand{\cG}{\mathcal{G}}

\newcommand{\bep}{\boldsymbol{\varepsilon}}
\newcommand{\ep}{\varepsilon}

\newcommand{\0}{\mathbf{0}}
\newcommand{\y}{\mathbf{y}}
\newcommand{\F}{\mathbf{F}}

\newcommand{\al}{\alpha}
\newcommand{\be}{\beta}

\newcommand{\imp}{\rightarrow}
\newcommand{\eqi}{\longleftrightarrow}
\newcommand{\ra}{\rangle}
\newcommand{\la}{\langle}

\newcommand{\st}{\operatorname{\mathbf{st}}}

\newcommand{\fin}{\operatorname{\mathbf{fin}}}
\newcommand{\dom}{\operatorname{dom}}
\newcommand{\ran}{\operatorname{ran}}

\newcommand{\shh}{\operatorname{\mathbf{sh}}}

\newcommand{\ZF}{\mathbf{ZF}}
\newcommand{\ZFC}{\mathbf{ZFC}}

\newcommand{\SPOT}{\mathbf{SPOT}}

\newcommand{\SCOT}{\mathbf{SCOT}}

\newcommand{\BST}{\mathbf{BST}}
\newcommand{\IST}{\mathbf{IST}}

\newcommand{\T}{\mathbf{T}}

\newcommand{\ACC}{\mathbf{ACC}}
\newcommand{\ADC}{\mathbf{ADC}}
\newcommand{\SP}{\mathbf{SP}}

\newcommand{\N}{\mathbf{O}}
\newcommand{\AC}{\mathbf{AC}}

\newcommand\astx{{}^\ast\!X}

\subjclass[2020]{Primary 26E35; Secondary 34A12}

\begin{document}
\title {Peano and Osgood theorems via effective infinitesimals}

\author{Karel Hrbacek} \address{Department of Mathematics\\ The City
  College of CUNY\\ New York, NY 10031\\} 
\email{khrbacek@icloud.com}

\author{Mikhail G. Katz} \address{Department of Mathematics\\ Bar Ilan
  University\\ Ramat Gan 5290002 Israel\\}
\email{katzmik@math.biu.ac.il}


\date{\today}

\begin{abstract}
We provide choiceless proofs using infinitesimals of the global
versions of Peano's existence theorem and Osgood's theorem on maximal
solutions.  We characterize all solutions in terms of infinitesimal
perturbations.  Our proofs are more effective than traditional
non-infinitesimal proofs found in the literature.  The background
logical structure is the internal set theory $\SPOT$, conservative
over $\ZF$.
\end{abstract}

\maketitle

\tableofcontents

\section{Introduction}

Nonstandard analysis (NSA) is sometimes criticized for its implicit
dependence on strong forms of the Axiom of Choice ($\AC$).  Indeed,
if~$\ast$ is the mapping that assigns to each $X \subseteq \bbN$ its
nonstandard extension $\astx$, and if $\nu \in {}^\ast \bbN \setminus
\bbN$ is an unlimited integer, then the set $U = \{ X \subseteq \bbN
\,\mid\, \nu \in \astx\}$ is a nonprincipal ultrafilter over $\bbN$.
Of course strong forms of $\AC$, such as Zorn's Lemma, are a staple of
modern set-theoretic mathematics, but it is undesirable to have to
rely on them for results in ordinary mathematics dealing with Calculus
or differential equations (see Simpson~\cite{Si} for a discussion of
the distinction between \emph{set-theoretic} and \emph{ordinary}
mathematics). The traditional proofs of most theorems in ordinary
mathematics are \emph{effective}: they do not use $\AC$.%
\footnote{In this paper the word \emph{effective} means \emph{without
the Axiom of Choice}.  In reverse mathematics, constructive
mathematics and other areas, it usually has more restrictive meaning.}
A few results, such as the equivalence of the $\ep$-$\delta$
definition and the sequential definition of continuity for functions
$f: X \subseteq \bbR \to \bbR$, require weak forms of $\AC$, notably the Axiom of
Countable Choice ($\ACC$) or the stronger Axiom of Dependent Choice
($\ADC$). These weak forms are generally accepted in ordinary
mathematics; they do not imply the strong consequences of $\AC$ such
as the existence of nonprincipal ultrafilters or the Banach--Tarski
paradox (see Jech~\cite{J1}, Howard and Rubin~\cite{Ho}).  We refer to
such proofs as \emph{semi-effective}.

An answer to the above criticism of NSA is offered by recent
developments in the axiomatic/syntactic approach that dates back to
the work of Hrbacek~\cite{Hr78} and Nelson~\cite{nelson}.  A number of
axiomatic systems for NSA have been proposed, of which Nelson's $\IST$
is the best known. We refer to Kanovei and Reeken's
monograph~\cite{KR} for a comprehensive discussion of such axiomatic
frameworks.  An accessible introduction to $\IST$ is
Robert~\cite{Robert}.

The theory $\IST$ includes the axioms of $\ZFC$, so one could ask
whether the dependence on $\AC$ could be avoided by deleting $\AC$
from the axioms constituting $\IST$.  It turns out that in the
resulting theory one can still prove the existence of nonprincipal
ultrafilters, by an argument similar to the one given above for the
model-theoretic approach (see~\cite{H1} and the paragraph following
Lemma~\ref{equivalent} below).

In~\cite{HK} the authors have developed an axiomatic system for NSA
with the acronym $\SPOT$, a subtheory of $\IST$.  The theory $\SPOT$
is a conservative extension of $\ZF$.  This means that every statement
in the $\in$-language provable in $\SPOT$ is provable already in
$\ZF$. In particular, $\AC$ and the existence of nonprincipal
ultrafilters are not provable is $\SPOT$, because they are not
provable in $\ZF$.  A stronger theory $\SCOT$ which is a conservative
extension of $\ZF + \ADC$ is also considered there.  Hence proofs in
$\SPOT$ are effective, and proofs in $\SCOT$ are semi-effective.

Some examples of constructions in nonstandard analysis formalized in
these theories are given in~\cite{HK}.  In particular, it is shown
there how the Riemann integral can be defined in $\SPOT$ using
partitions into infinitesimal subintervals, and the countably additive
Lebesgue measure in $\SCOT$ using counting measures. The expository
article~\cite{HK3} presents in $\SCOT$ various nonstandard arguments
related to compact sets and continuity.

In Section~\ref{starzf} we state the axioms of $\SPOT$, list some of
their consequences, and prove a stronger version of the Standard Part
principle $\SP$ that is crucial in the preliminary
Section~\ref{examples}.

In Sections~\ref{peano} - \ref{osgoodsection} we give nonstandard
proofs in $\SPOT$ of the global versions of Peano's and Osgood's
theorems concerning the existence of solutions of ordinary
differential equations. While the nonstandard approach using Euler
approximations with an infinitesimal step that we employ is well known
for local solutions (see e.g.~\cite{A}, p.\;30), we offer three
innovations:
\begin{itemize}
\item
The axiomatic system $\SPOT$ enables us to use infinitesimal methods
without the underlying assumption of the existence of nonprincipal
ultrafilters or any other strong form of $\AC$.
\item
We construct global, ie, noncontinuable, solutions rather than local
solutions.
\item
Traditional proofs of the existence of noncontinuable solutions
typically depend on $\ADC$; see Remark~\ref{others}.  By contrast, our
proof does not assume any form of $\AC$ at all.
\end{itemize}

We first prove (Theorem~\ref{Theorem}) that every infinitesimal
perturbation $\bep$ determines a unique global solution $y_{\bep}$
(some or all of these solutions may be the same).  We next prove
(Lemma~\ref{tags}) that every solution that is not global is a
restriction of some $y_{\bep}$. Hence every solution is either global
or can be extended to a global one (Corollary~\ref{extension}) and
every global solution is of the form $y_{\bep}$ for some infinitesimal
perturbation~$\bep$ (Theorem~\ref{perturb}).  Finally we state the
global Osgood's theorem (Theorem~\ref{osgood}). The proof shows first
that there is a local maximal solution (Lemma~\ref{L2} and the last
part of the sentence that precedes it).  The last paragraph of the
proof obtains the global maximal solution as the union of all local
ones.


\section{Theory {\bf SPOT}} 
\label{starzf}

By an $\in$-language we mean the language that contains a binary
membership predicate $\in$ and is enriched by defined symbols for
constants, relations, functions and operations customary in
traditional mathematics. For example, it contains names $\bbN$ and
$\bbR$ for the sets of natural and real numbers; they are viewed as
defined in the traditional way ($\bbN$ is the least inductive set,
$\bbR$ is defined in terms of Dedekind cuts or Cauchy sequences). The
symbols $<, +$ and $\times$ denote the ordering, addition and
multiplication of real numbers, and so on without further explanation.
The classical theories $\ZF$ and $\ZFC$ are formulated in the
$\in$-language.

 The language of $\SPOT$ contains an additional unary predicate 
$\st$.  $\SPOT$ is a subtheory of $\textbf{IST}$ and its bounded version 
$\textbf{BST}$ (see~\cite{KR}).  
We use
$\forall$ and $\exists$ as quantifiers over sets and $\forall^{\st}$
and $\exists^{\st}$ as quantifiers over standard sets.  The theory
$\SPOT$ has the following axioms.

\bigskip
$\ZF$ (Zermelo - Fraenkel Set Theory)

\bigskip
$\T$ (Transfer) 
Let $\phi$ be an $\in$-formula with standard parameters. Then
$$\forall^{\st} x\; \phi(x)  \imp  \forall x\; \phi(x).$$

$\N$ (Nontriviality)  \quad
$\exists \nu \in \bbN\; \forall^{\st} n \in \bbN\; (n \ne \nu)$.

\bigskip
$\SP'$  (Standard Part)
$$ \forall A \subseteq \bbN \; \exists^{\st} B \subseteq \bbN \;
\forall^{\st} n \in \bbN \; (n \in B \eqi n \in A). $$

\bigskip

The theory $\SPOT$ proves the following results (see~\cite{HK}).

\begin{lemma}
\label{downclosed} 
Standard natural numbers precede all nonstandard ones:
$$\forall^{\st} n \in \bbN \; \forall m \in \bbN \;(m < n \imp \st(m) ).$$
\end{lemma}

Note that $\{0, 1, \ldots, n-1\}$ is a finite set for every $n \in \bbN$; it is nonstandard when $n$ is nonstandard.

\begin{lemma}[Countable Idealization]
\label{countideal} 
Let $\phi$ be an $\in$-formula with arbitrary parameters.
$$\forall^{\st}  n \in \bbN\;\exists x\; \forall m \in \bbN\; (m \le n \;\imp  \phi(m,x))\eqi 
\exists x \; \forall^{\st} n \in \bbN \; \phi(n,x).$$
\end{lemma} 

The  dual form of Countable Idealization is
$$
\exists^{\st}  n \in \bbN\;\forall x\; \exists m \in \bbN\; (m \le n \,\wedge\,  \phi(m,x))
\eqi \forall x \; \exists^{\st} n \in \bbN \; \phi(n,x) .$$

Countable Idealization easily implies the following more familiar form. We use  $\forall^{\st \fin}$ and $\exists^{\st \fin}$ as quantifiers
over standard finite sets.

\begin{corollary} \label{corcountideal}
Let $\phi$ be an $\in$-formula with arbitrary parameters.  
 For every standard countable
set~$A$
\[
\forall^{\st \fin} a \subseteq A \, \exists x \, \forall y \in a\;
\phi(x,y) \eqi \exists x\, \forall^{\st} y \in A\; \phi(x,y) .
\]
\end{corollary}

The axiom $\SP'$ is often stated and used in the
form
\begin{equation}
 \forall x \in \bbR \;(x \text{ limited } \imp \exists^{\st} r \in \bbR \;(x \approx r)) \tag{$\SP$}
\end{equation}
where $x$ is \emph{limited} iff $|x| \le n$ for some standard $n \in \bbN$, and $x \approx r$ iff $|x - r| \le 1/n $ for all standard $n \in \bbN$, $n \neq 0$; $x$ is \emph{infinitesimal} if $x \approx 0 \,\wedge\, x \neq 0$. 
The unique standard real number $r$ in $\SP$  is called the \emph{standard part of} $x$ or the \emph{shadow of} $x$; notation $r = \shh(x)$.

\bigskip
We have the following equivalence.
\begin{lemma}\label{spandsp'}
The statements $\SP' $ and $ \SP$ are equivalent (over the theory $\ZF
+ \N + \T$).
\end{lemma}

$\SP'$ can  also be reformulated as an axiom schema (\emph{Countable
  Standardization for $\in$-formulas}):
\[
\begin{aligned}
\text{Let $\phi$ } & \text{be an $\in$-formula with arbitrary
  parameters. Then} \\& \exists^{\st} S\; \forall^{\st} n \; (n \in S
\eqi n \in \bbN \,\wedge\, \phi(n)).
\end{aligned}
\tag{$\SP''$}
\]

\begin{lemma}\label{equivalent}
The statement $\SP' $ and the schema $ \SP''$ are equivalent (over the theory
$\ZF + \N + \T$).
\end{lemma}

\begin{proof}
 Apply  $\SP'$ to the set  $A = \{ n \in \bbN \,\mid\, \phi(n) \}$ ($A$ exists because $\phi$ is an $\in$-formula).
\end{proof}

Standardization in full strength, as postulated in $\IST$, $\BST$, etc., implies the existence of nonprincipal ultrafilters over $\bbN$: take a nonstandard $\nu \in \bbN$ and let $U$ be the standard subset of $\cP(\bbN)$ such that $\forall^{\st} X \subseteq \bbN\; (X \in U \eqi \nu \in X)$.  Nonetheless, two important special cases of Standardization can be proved in $\SPOT$.

The scope of Countable Standardization can be expanded to a larger class of formulas.

\begin{definition}
An $\st$-$\in$-formula  $\Phi(v_1, \ldots,v_r)$ is $\st$-\emph{prenex} if it is of the form 
$$\mathsf{Q}^{\st} u_1\ldots \mathsf{Q}^{\st} u_s \,\psi(u_1,\ldots ,u_s,v_1, \ldots,v_r)$$
 where $\psi $ is an $\in$-formula and each $\mathsf{Q}$ stands for $\exists$ or $\forall $. 

In other words, all occurrences of $\forall^{\st}$ or $\exists^{\st}$  in $\Phi$ 
appear before all occurrences of $\forall$ or $\exists$.

We use $\forall^{\st}_\bbN \,u\ldots$ and $\exists^{\st}_\bbN\, u \ldots$ as quantifiers over standard natural numbers; i.e. as shorthand for respectively 
$\forall u\,( u \in \bbN \,\wedge\, \st(u) \imp \ldots)$ and $\exists u \,(u \in \bbN \,\wedge\, \st(u) \,\wedge\, \ldots)$. 

An $\st_{\bbN}$-\emph{prenex} formula is a formula  of the form 
\[
\mathsf{Q}^{\st}_{\bbN} u_1\ldots \mathsf{Q}^{\st}_{\bbN} u_s \,\psi(u_1,\ldots u_s, v_1, \ldots,v_r)
\]
where $\psi$ is an $\in$-formula.
\end{definition}

The theory $\SPOT$ proves the following stronger version of Countable Standardization that is used repeatedly in this paper.

\begin{proposition} \emph{(Countable Standardization for $\st_{\bbN}$-prenex formulas)}
 Let $\Phi$ be an $\st_{\bbN}$-prenex formula with arbitrary parameters. Then
\[
\exists^{\st} S\; \forall^{\st} n \; (n \in S \eqi n \in \bbN
\,\wedge\, \Phi(n)).
\]
\end{proposition}

Of course, $\bbN$ can be replaced by any standard countable set. 

\begin{proof}
We give the argument for a typical case
$$\forall^{\st}_\bbN  u_1\,\exists^{\st}_\bbN u_2\,\forall^{\st}_\bbN u_3 \;\psi(u_1,u_2,u_3, v).$$
By $\SP''$ there is a standard set $R$ such that for all standard $n_1, n_2, n_3, n$
$$\la  n_1, n_2, n_3, n \ra \in R \eqi \la n_1, n_2, n_3, n \ra \in \bbN^4 \,\wedge\, \psi(n_1, n_2, n_3,n).$$ 
We let $R_{n_1, n_2, n_3} = \{n \in \bbN \,\mid\, \la n_1, n_2, n_3, n \ra \in R\}$ and 
$$S = \bigcap_{n_1 \in \bbN}\, \bigcup_{n_2 \in \bbN}\,\bigcap_{n_3 \in \bbN} \,R_{n_1, n_2, n_3} .$$
Then $S$ is standard and for all standard $n$:
\[
\begin{aligned}
n \in S & \eqi 
 \forall n_1 \in \bbN\,\exists n_2 \in \bbN\,\forall n_3 \in \bbN\; (n \in R_{n_1, n_2, n_3} )
\\&\eqi
\text{(by Transfer) }
\forall^{\st}_{\bbN} n_1 \,\exists^{\st}_{\bbN} n_2 \;\forall^{\st}_{\bbN} n_3 \; (n \in R_{n_1, n_2, n_3}) 
\\&\eqi
\text{ (by definition of $R$) }
  \forall^{\st}_{\bbN} n_1 \,\exists^{\st}_{\bbN} n_2 \;\forall^{\st}_{\bbN} n_3 \;\psi(n_1, n_2, n_3,n)
\\&\eqi 
\Phi(n). 
\end{aligned}  
\]
\end{proof}

The second special case of Standardization involves $\st$-prenex  formulas with only the standard parameters.

\begin{lemma}\label{stinternal}
Let $\Phi (v_1,\ldots,v_r)$ be an $\st$-prenex  formula
with standard parameters. 
Then  $\quad$
$ \forall^{\st} S\;\exists^{\st} P\; \forall^{\st} v_1,\ldots,v_r\;$
$$  \la v_1,\ldots,v_r \ra \in P \eqi
\la v_1,\ldots,v_r \ra \in S \; \wedge \; \Phi (v_1,\ldots,v_r) .$$

\end{lemma}

\begin{proof}
Let $\Phi (v_1\ldots, v_r)$ be  $\mathsf{Q}_1^{\st} u_1\ldots  \mathsf{Q}_s^{\st} u_s\, \psi(u_1,\ldots ,u_s,v_1, \ldots,v_r)$
and  let $\phi (v_1\ldots, v_r)$ be
 $\mathsf{Q}_1  u_1\ldots  \mathsf{Q}_s u_s\, \psi(u_1,\ldots ,u_s, v_1, \ldots,v_r)$.
Since $\Phi$ has standard parameters, 
$ \Phi (v_1\ldots,v_r) \eqi \phi (v_1\ldots,v_r)$ holds for all standard $ v_1\ldots,v_r$ 
by the Transfer principle.

The  set $P = \{ \la v_1,\ldots,v_r \ra  \in S \, \mid \, \phi( v_1,\ldots,v_r)\}$ exists by the Separation Principle of $\ZF$, it is standard, and has the required property.
\end{proof}

 \begin{remark} \label{forallz}
This result has twofold importance:
\begin{enumerate}
\item
The meaning of every predicate that for standard inputs is defined by an $\st$-prenex  formula  $\mathsf{Q}_1^{\st} u_1\ldots  \mathsf{Q}_s^{\st} u_s\, \psi$   with standard parameters is automatically extended to all inputs, where it is given by the  $\in$-formula $\mathsf{Q}_1 u_1\ldots  \mathsf{Q}_s u_s\, \psi$.
\item
Standardization holds for all $\in$-formulas  with additional predicate symbols, as long as all these  additional predicates are defined by $\st$-prenex formulas with standard parameters.
\end{enumerate}
\end{remark}


\section{Two examples}
\label{examples}

Formulas that occur in practice are usually not in the $\st$-prenex form, but they can often be converted to it using Countable Idealization.

\begin{definition}\label{defriemann}
[\emph{Integral of continuous functions}]

We fix a positive infinitesimal $h$ and the corresponding ``hyperfinite  line'' 
$\{ x_i \,\mid\, i \in \bbZ\}$ where $x_i = i\cdot h$. Let $f$ be a standard real-valued function continuous on the standard interval $[a, b ]$. Let $i_a, i_b$ be such that $ i_a \cdot h- h < a \le  i_a \cdot h$ and 
$i_b \cdot h < b \le i_b\cdot h + h$. We define
\begin{equation}
\int_a^b f(x)\, dx = \shh \left( \Sigma_{i = i_a}^{ i_{b} } f(x_i)\cdot h  \right).
\end{equation} 
It is easy to show that the value of the integral does not depend on the choice of $h$. 
\end{definition}

\begin{lemma}\label{riemann}
There is an $\st_{\bbN}$-prenex  formula $\Phi(v_1, v_2,v_3,v_4)$ such that 
$\int_a^b f(x)\, dx =r  \eqi \Phi(f,a,b,r)$ holds
for all standard $f,a,b,r$.
\end{lemma}

\begin{proof}
For standard $f, a,b, r$ we have $\int_a^b f(x)\, dx  = r $ \; iff 
$$
\forall h \,\left[  \forall^{\st}_{\bbN} n \;(|h| < \tfrac{1}{n}) \imp\,\forall^{\st}_{\bbN} m \left( |\Sigma_{i = i_a}^{  i_{b}  } f(x_i)\cdot h  - r | < \tfrac{1}{m}\right)\right] $$
 (it is understood  that $h$, $n$, $m$ are not $0$).
This expression can be rewritten as
$$ \forall h  \; \forall^{\st}_{\bbN} m\; \exists^{\st}_{\bbN} n\,
\left[  |h| \ge \tfrac{1}{n}  \,\vee\,   |\Sigma_{i = i_a}^{  i_{b}  } f(x_i)\cdot h  - r | < \tfrac{1}{m} \right].$$
We swap the outmost universal quantifiers and apply the dual version of Countable Idealization 
 (Lemma~\ref{countideal}) to get
$$ \forall^{\st}_{\bbN} m\; \exists^{\st}_{\bbN} n\,\forall h  \; \exists k \le n\; 
\left[  |h| \ge \tfrac{1}{k}  \,\vee\,   |\Sigma_{i = i_a}^{  i_{b}  } f(x_i)\cdot h  - r | < \tfrac{1}{m} \right],$$
which is an $\st_{\bbN}$-prenex formula, clearly equivalent to 
$$ \forall^{\st}_{\bbN} m\; \exists^{\st}_{\bbN} n\,\forall h  \;  
\left[  |h| \ge \tfrac{1}{n}  \,\vee\,   |\Sigma_{i = i_a}^{  i_{b}  } f(x_i)\cdot h  - r | < \tfrac{1}{m} \right].$$
\end{proof}

One can now use Standardization for $\st$-prenex formulas with standard parameters to conclude that, for example, for every standard $f$, $a$ there exists a standard function $F$ such that $F(z) =  \int_a^z f(x)\, dx $ for all  standard $z \in [a,b]$. By Remark~\ref{forallz} (1), the last equation holds for all $z \in [a,b]$.
Of course, the usual arguments show that the above definition of the integral agrees with the traditional $\epsilon$-$\delta$ one for all standard $f, a, b, r$.

The following observation is crucial for the proof of Proposition~\ref{propW}.

\begin{lemma}\label{defofZ}
Let $w$ be a  function, $\dom w = D_w \subseteq \bbR$ and  $\ran w  \subseteq \bbR$.   Then the formula
$$\Psi(x,y): \quad \exists \al \in D_w \;[x \approx \al \,\wedge\, (y \approx w(\al) \,\vee\, y \ge w(\al))]$$ is equivalent to an 
$\st_{\bbN}$-prenex formula (with the parameter $w$).
\end{lemma}

\begin{proof}
The formula $\Psi(x,y)$ can be written as 
$$\exists \al \in D_w\;[\forall^{\st}_{\bbN} i\, (|x - \al| < \tfrac{1}{i+1})\,\wedge\, (\forall^{\st}_{\bbN} j\,
(|y - w(\al)| <  \tfrac{1}{j+1}) \,\vee\, y \ge w(\al))],$$
which is equivalent to
$$\exists \al \in D_w\;\forall^{\st}_{\bbN} i\; \forall^{\st}_{\bbN} j\, [ (|x - \al| <  \tfrac{1}{i+1})\,\wedge\, 
(|y - w(\al)| <  \tfrac{1}{j+1} \,\vee\, y \ge w(\al))].$$
This is equivalent to 
$$\exists \al \in D_w\;\forall^{\st}_{\bbN} n\, \ [ (|x - \al| <  \tfrac{1}{n+1})\,\wedge\, 
(|y - w(\al)| <  \tfrac{1}{n+1} \,\vee\, y \ge w(\al))]$$
(let $n = \min{\{i,j}\}$), and finally (Countable Idealization,  Lemma~\ref{countideal}) to the $\st_{\bbN}$-prenex formula
$$\forall^{\st}_{\bbN} n\, \exists \al \in D_w\; \forall m \le n\ [ (|x - \al| <  \tfrac{1}{m+1})\,\wedge\, 
(|y - w(\al)| <  \tfrac{1}{m+1} \,\vee\, y \ge w(\al))].$$
The last formula of course simplifies to 
$$\forall^{\st}_{\bbN} n\, \exists \al \in D_w\; [ (|x - \al| <  \tfrac{1}{n+1})\,\wedge\, 
(|y - w(\al)| <  \tfrac{1}{n+1} \,\vee\, y \ge w(\al))]. $$
\end{proof}

\begin{definition}
Let $w$ be a  function,   $\dom w = D_w \subseteq I$ where $I \subseteq \bbR$ is a standard interval, and $\ran w  \subseteq \bbR$.
\begin{itemize}
\item
The function  $w$ is \emph{densely defined on}  $I$ if for every standard $x \in I$ there is $\al \in D_w$ such that $\al \approx x$.
\item
 The function  $w$ is (uniformly)
S-\emph{continuous} if for $\al, \be \in D_w$, $\al \approx \be $ implies $w(\al) \approx w(\be)$.
\end{itemize}
\end{definition}

\begin{lemma}
A function $w$ is S-continuous iff for every standard $\epsilon > 0$
there is a standard $\delta > 0$ such that for $\al, \be \in D_w$, $|\al - \be| < \delta$ implies $|w(\al) -w(\be)| < \epsilon$.
\end{lemma}

\begin{proof}
The usual arguments work in $\SPOT$; see e.g.~\cite{HK3}.
\end{proof}

The next proposition follows immediately from the Standardization
principle of $\IST$ or $\BST$, but to prove it in $\SPOT$ we need to
consider an approximation to the set $W$ on the rationals, to which we
can apply Countable Standardization for $\st_{\bbN}$-prenex formulas.

\begin{proposition}\label{propW}
If $w$ is S-continuous and densely defined  on $I$, then there is a standard function $W$ such that, for all standard $x, y \in \bbR$,
$\la x, y \ra \in W$ if and only if $x \approx \al$ and $y \approx w(\al)$ for some $\al \in D_w$.
\end{proposition}

 The proof of Proposition~\ref{propW} appears below, following the proof of Lemma~\ref{loccont}.

\begin{definition}\label{defZ}
 The existence of the standard set 
 \[
Z =\, {}^{\st} \{ \langle q, r \rangle \in (I \cap \bbQ) \times \bbQ \mid\, \exists \al \in D_w\, 
[q \approx \al \,\wedge\, (r \approx w(\al) \,\vee\, r \ge w(\al)] \}
\]
is justified in Lemma~\ref{defofZ}.

For $q \in I \cap \bbQ$ let $Z_q = \{ r \in \bbQ \,\mid\, \la q, r \ra \in Z\}$ and $W_0(q) = \inf Z_q$, if it exists
(it can happen that $Z_q = \emptyset $ or $Z_q = \bbQ$, in which cases $W_0(q)$ is undefined).
Finally, let $W$ be the closure of (the graph of) $W_0$.
We show below that the standard set $W$ has the property from Proposition~\ref{propW}.
\end{definition}

\begin{lemma}\label{loccont}
If $q \in I \cap \bbQ$ is standard, then $q \in \dom W_0$ if and only if there exists $\al \in D_w$ such that $\al \approx q$ and $w(\al)$ is limited.  If this is the case, then  $W_0(q) = \shh(w(\al))$.
\end{lemma}

\begin{proof}
 If $\al, \be \in D_w$, $q \approx \al$ and $q \approx \be$, then $w(\al) \approx w(\be)$,
 so we have $Z_q = {}^{\st}\{ r \in \bbQ \,\mid\,  r \approx w(\al) \,\vee\, r \ge w(\al )\}$, independently of the choice of $\al$.
If $w(\al)$ is limited, then $  \inf Z_q = \shh(w(\al))$.
If $w(\al)$ is unlimited, then $Z_q = \emptyset$ or $Z_q = \bbQ$, so $W_0(q)$ is undefined.
\end{proof}

\begin{proof}[Proof of Proposition~\ref{propW}]
Assume that $x \approx \al $ and $ y \approx w(\al)$ for $\al \in D_w$.
Given any standard $\epsilon > 0$, take a standard $\delta>0 $
witnessing S-continuity of $w$, a standard $q \in \bbQ \cap I$ such
that $|x - q| < \min\{\delta, \epsilon\}$ and some $\be \approx q$,
$\be \in D_ w$.  Then $|\al - \be| < \delta$, and hence $|w(\al) -
w(\be)| < \epsilon$.  It follows that $w(\be)$ is limited.  By
Lemma~\ref{loccont}, $w(\be) \approx W_0(q)$, so $|x-q|<\epsilon$ and
$|y - W_0(q) |< \epsilon$.  This shows that $\la x, y \ra \in W $.

Conversely, if $\la x, y \ra \in  W $, then for every standard $\epsilon > 0$ there is 
$q\in  \dom W_0$ such that $|x - q | < \epsilon $ and $|y - W_0(q)| < \epsilon$.
Let $\al \in D_w $, $\al \approx q$; then $w(\al) \approx W_0(q)$, $|x -\al|< \epsilon$ and 
$|y - w(\al)|< \epsilon$.
By Countable Idealization  (Lemma~\ref{countideal}) there is $\al \in D_w$ such that for all standard $\epsilon > 0$ we have 
$|x -\al|< \epsilon$ and  $|y- w(\al)|< \epsilon$.
Then $x \approx \al$ and $y \approx w(\al)$.
\end{proof}


\section{Peano's  Existence Theorem in {\bf SPOT}}
\label{peano}

\begin{theorem}[Global Peano's  Theorem]\label{Theorem}

Let $F: [0, \infty) \times \mathbb{R} \to \mathbb{R}$ be a continuous
  function.  There is an interval $[0, a) $ with $0 < a \le \infty$
    and a function $y: [0,a) \to \mathbb{R}$ such that
\begin{equation}
\tag{$\ast$}
\label{star}
\qquad y(0)= 0,\quad
\quad y'(x) =F(x,y(x)) \quad
\end{equation}
holds for all $ x \in [0, a)$, and if $a \in \mathbb{R}$ then
  $\lim_{x\imp a^-} y(x) = \pm \infty$.
\end{theorem}

Here and elsewhere, if $c \in \bbR$ is an endpoint of an interval $I =
\dom y$, $y'(c)$ is the appropriate one-sided derivative of $y$ at
$c$.  We call a solution of the initial value problem \eqref{star}
that cannot be continued to any interval $[0, a')$ with $a' > a$ a
  \emph{global} solution.

We generalize the familiar construction of Euler approximations with an infinitesimal step by allowing infinitesimal perturbations.
This is a variation on an idea in Birkeland and Normann~\cite{BN}
(the main difference being that we perturb the construction of the solution, while Birkeland and Normann  perturb the function $F$).

We will prove the theorem for standard $F$; the stated result follows by Transfer.
The construction proceeds as follows.

Let $N$ be a positive  unlimited integer and $h=1/N$.
We fix $x_0 \ge 0$, $x_0 \approx 0$, $y_0 \approx 0$,   and
let $x_k=x_0 + k\cdot h$ for
$k=0,\ldots,N^2$.

\begin{definition}
An \emph{infinitesimal perturbation} is a sequence $\bep = \la \ep_k
\,\mid\, k < N^2\ra $ such that each $\ep_k \approx 0$; we let $\ep =
\max\{|\ep_k |\,\mid\, k < N^2\}$.
\end{definition}

The concept is not needed for the proof of Theorem~\ref{Theorem},
where the simplest choice $\ep_k = 0 $ for all $k$ suffices, but it is
used for its generalization in Section~\ref{perturbs}.

We define $y_k$  recursively: 
\[
y_{k+1}=y_k + (F(x_k,y_k)+ \ep_k) \cdot h \quad \text{ for } k < N^2.\]
Observe that
\[
y_{\ell} = y_k + \sum_{i=k}^{\ell - 1} (F(x_i, y_i) + \ep_i)\cdot h,
\quad \text{for any $k < \ell \le N^2$.}
\]

We next define
\begin{equation}
\tag{$\ast\ast$}
\label{starstar}
  Y =\, {}^{\st} \{ \langle x, y \rangle \in [0, \infty) \times
    \mathbb{R}\,\mid\, x \approx x_k \,\wedge\, y \approx y_k \text{
      for some } k < N^2\}. 
\end{equation}
The existence of $Y$  in $\SPOT$ follows from Proposition~\ref{propW} (let $I = [0, \infty)$ and 
$w (x_k) = y_k$ for $0 \le k < N^2 $).
The strategy for the rest of the proof  is to show that $Y$ is a (graph of a) continuous function  defined on an open subset of $[0, \infty)$, and the restriction $y$ of $Y$ to the connected component of its domain containing $0$ has the required properties.

\begin{lemma}\label{techlemma}
Let $\la x, y\ra \in [0, \infty)  \times \mathbb{R} $ be standard and $x_p-h <x \le x_p$, $y \approx y_p$ for some $p < N^2$.
There exist standard $d, e , M> 0$ such that
$y_k \in [y-d, y+ d]$  for all $x_k\in [x, x+ e)$ and $|y_k - y_\ell| \le (M + \ep)\cdot |x_k - x_\ell|$ 
 for all $x_k, x_\ell \in [x, x+e)$.
In particular, if $x_k \approx x_\ell \approx x$ then $y_k \approx y_\ell$.\\
If $x >0$ then $[x,x + e)$ can be replaced by  $(x-e, x+e)$.
\end{lemma}

\begin{proof}
By continuity of $F$ at  $\langle x, y\rangle$ there exist standard $c, d  , M > 0$ such that $|F(t,s)| \le M$ holds for all $\langle t, s \rangle \in  [x, x +c] \times [y-d, y +d]$; if $x>0$,we can assume also $c \le x$.
Fix a standard $e$ such that  $0 < e < \min\{ c, \,d/(M+1)\}$.

We prove by induction on $k$  that $$k \ge p \,\wedge\, x_k < x + e \imp |y_k - y_p| \le (M + \ep) \cdot| x_k - x_p| .$$

The case $k = p$ is clear. If the claim is true for $k$ and $x_{k+1} < x + e$, we have 
$|y_k- y_p| \le(M + \ep)\cdot |x_k- x_p|<  (M+1) \cdot e  \le d$ and hence the point $\langle x_k, y_k\rangle \in [x, x+c] \times [y-d, y+d]$. 
Now $|F(x_k, y_k)| \le M$, so 
$|y_{k+1} 
- y_k| \le (|F(x_k, y_k)| + |\ep_k|)\cdot h \le    (M + \ep) \cdot h$ and 
\[
\begin{aligned}
|y_{k+1} - y_p| &\le |y_{k+1} - y_k| + |y_k - y_p| \\&\le (M + \ep) \cdot h
+ (M + \ep) \cdot |x_k - x_p| \\&= (M+ \ep)\cdot |x_{k+1}-x_p|.
\end{aligned}
\]
Finally, $|y_\ell - y_k| \le \sum_{i=k}^{\ell - 1} (|F(x_i, y_i)|+\ep)
\cdot h \le (M+\ep) \cdot |x_\ell - x_k| $.

If $x> 0$, a symmetric ``backward'' argument shows that the statement
holds also on the interval $(x-e, x]$.
\end{proof}

It follows easily that $Y$ as in \eqref{starstar} is the graph of a
real function.  If $\la x, y_1\ra$, $ \la x, y_2\ra \in Y$ are
standard, then there are $k$ and $\ell$ such that $\la x, y_1 \ra
\approx \la x_k, y_k\ra$ and $\la x, y_2 \ra \approx \la x_\ell,
y_\ell\ra$.  Then $x_k \approx x_\ell\approx x$ and $y_1 \approx y_k
\approx y_\ell \approx y_2 $ and we conclude that $y_1 = y_2$. Hence
$Y$ is the graph of a function, by Transfer.  From now on we write
$Y(x)$ for the value of $Y$ at $x \in \dom Y$.

\begin{lemma}
The domain of the function $Y$ is an open subset of $[0, \infty )$ containing $0$, $Y$ is continuous on $\dom Y$, and $Y'(x) = F(x, Y(x))$ holds for $x \in \dom Y$.
\end{lemma}

\begin{proof}
Clearly $0 \in \dom Y$. 
If $x \in \dom Y$ is standard and $y = Y(x)$, Lemma~\ref{techlemma} gives an interval $I = (x-e, x+e) $ (or $I = [0, e)$)
such that $x_k \in I$ implies $y_k \in [y-d, y+d]$, hence $y_k$ is limited and $Y(\shh(x_k)) = \shh(y_k)$ is defined. If $u \in I$ is standard, $u =\shh( x_k)$ holds for some $x_k \in I$. Hence $Y(u) \in [y-d, y+d]$ is defined for all standard $u \in  I$, and by Transfer, the same holds for all $u \in I$.

For the proof of  continuity at a standard $x\in \dom Y$ let $I$ be as above, $\epsilon  > 0$ be standard and $\delta = \epsilon /(M+1)$.
If $z \in I$ is standard and $ |x-  z| < \delta$, then there are $k$ and $\ell$ such that $x \approx   x_k $ and  $ z\approx   x_\ell$; moreover, 
$Y(x) \approx y_k$ and $Y(z) \approx y_\ell$.  
We have $|Y(x) - Y(z) |\approx |y_\ell - y_k| $ and 
$|y_k - y_\ell| \le   (M+\ep) \cdot |x_k- x_\ell| $,
where $|x_\ell - x_k| \approx |x -z| < \delta$. 
It follows  that $|Y(x) - Y(z)| \le (M+1)\cdot \delta = \epsilon$.
As usual, Transfer gives continuity for all $x \in \dom Y$.

It remains to prove that $Y'(x) = F(x, Y(x))$ holds for $x \in \dom Y$.
Let $I$ be as above, $x, z\in I$  be standard and wlog.\ $x \le z$.
In the notation of the previous paragraph, we have
\[Y(z) - Y(x)  \approx  y_\ell - y_k = \sum_{i=k}^{\ell - 1} (F(x_i, y_i) +\ep_i)\cdot h   \tag{1}
\]
and
\[ \int_{x}^{z} F(t, Y(t)) \, dt \approx   \sum_{i=k}^{\ell - 1} F(x_i, Y(x_i)) \cdot h =
 \sum_{i=k}^{\ell - 1} (F(x_i, y_i) +\delta_i)\cdot h \tag{2}\]
where $\delta_i \approx 0$ for $k \le i < \ell$.
The relation $\approx$ in (2) follows from the nonstandard theory of  integration (see Definition~\ref{defriemann}) and the fact that $F(t, Y(t))$ is continuous on $I$.
The relation $=$ in (2) is justified as follows: 
Let $x^\ast = \shh(x_i)$ and $y^\ast = \shh(y_i)$; then $Y(x^\ast) \approx y^\ast$ by the definition of $Y$ and
$Y(x_i) \approx Y(x^\ast) $ by the continuity of $Y$.
The continuity of $F$ then gives $F(x_i,y_i) \approx F(x^\ast, y^\ast) \approx F(x_i, Y(x_i))$.

The formulas (1) and (2)  imply  $Y(z) - Y(x) \approx \int_{x}^{z} F(t, Y(t)) \, dt$, hence $Y(z) - Y(x) = \int_{x}^{z} F(t, Y(t)) \, dt$ as both sides are standard. By Transfer, the relationship holds for all $x$, $ z \in I$.
It remains to apply the Fundamental Theorem of Calculus.
\end{proof}

Let $[0, a)$, $a > 0$, be the connected component of the domain of $Y$ containing $0$.

\begin{lemma}
The function $Y$ satisfies $\lim_{x\imp a^-} Y(x) = \pm \infty$.
\end{lemma}

\begin{proof}
We prove that for every standard $r > 0$ there is a standard $\epsilon > 0$ such that for all standard $x$, $a  - \epsilon < x < a$ implies $ | y(x)| \ge r$.

Assume that the statement is false and fix a standard $r > 0$ such that for every standard 
 $n \in \bbN$ there is a standard $x \in (a -  \frac{1}{n}, a) $ such that $Y(x) \in (-r, r)$.
Hence for every standard $n \in \bbN$ there is $k < N^2$ such that $x_k \in  (a - \frac{1}{n}, \,a)$ and 
$y_k \in (-r,r)$
(take $\la x_k, y_k \ra \approx \la x, Y(x)\ra$).
By Countable Idealization  (Lemma~\ref{countideal}), there exists $o < N^2$ such that $y_o \in (-r,r)$ and $x_o \in  (a - \frac{1}{n}, \,a)$ holds for all standard $n >0$.
It follows that $ x_o \approx  a$; we let $b = \shh(y_o)$.
By the definition of $Y$ then $\la a, b \ra \in Y$,  and hence $a \in \dom Y$, contradicting the fact that $[0, a)$ is a connected component of the domain of $Y$.
\end{proof}

\begin{proof}[Conclusion of proof of Theorem~$\ref{Theorem}$]
Let $Y$ be the function defined by formula\;\eqref{starstar}.  The
proof of Theorem~\ref{Theorem} is now concluded by letting $y = Y
\upharpoonright [0, a).$ We write $ y_{\bep}$ when it is necessary to
  indicate the dependence of $y$ on the perturbation $\bep$.
\end{proof}

\begin{remark}
Note that the solution $y$ is determined by the choice of the starting
point $x_0$, $y_0$ and the infinitesimal perturbation $\bep$.  Thus we
can single out a particular global solution of \eqref{star} by fixing
$N$ and letting $x_0 = 0$, $y_0 = 0$ and $\ep_k = 0$ for all $k <
N^2$.
\end{remark}

\begin{remark}
There are obvious generalizations that do not require any additional nonstandard ideas.
For example, the two-sided version:

\emph{Let $F:  \mathbb{R}^2 \to \mathbb{R}$ be a continuous function.
For every $\la a, b\ra \in \mathbb{R}^2$
there is an interval $(a^-, a^+)$  with $-\infty \le a^- < a < a^+ \le \infty$
and a function $y: (a^-, a^+) \to \mathbb{R}$ such that
\[
y(a)= b,\quad
\quad y'(x) =F(x,y(x)) \quad
\text{ holds for all \;} x \in (a^-, a^+),
\]
and if $a^-$ and/or $a^+$ is  in $\mathbb{R}$, then $\lim_{x\imp (a^-)^+} y(x)= \pm \infty$ and/or \\
$\lim_{x\imp (a^+)^-} y(x) = \pm \infty$.}

The domain  $\mathbb{R}^2$ of $F$ can be replaced by an open set $D \subseteq \bbR^2$. 
One obtains a solution that tends to the boundary of $D$, in the sense that for every compact $K \subseteq D$ there is $c < a^{+}$ such that $y(x) \notin K$ holds for all $c < x < a^{+}$, and analogously for $a^{-}$.
\end{remark}

The method generalizes to systems of equations.
\begin{theorem}
Let $\F: \mathbf{D} \to \mathbb{R}^n$ be continuous on an open set
$\mathbf{D} \subseteq \bbR^{n+1}$ and $\la 0, \0 \ra\in \mathbf{D}$.
The initial value problem
\begin{equation}
\label{star5}
 \y(0)= \0, \quad
\quad \y'(x) =\F(x,\y(x)) \quad
 \tag{$\star$}
\end{equation}
has a noncontinuable solution.
\end{theorem}

\begin{proof}
For $\mathbf{u} = \la u_0, \ldots , u_{n-1}\ra$ and $\mathbf{v} = \la v_0, \ldots , v_{n-1}\ra$ in $\bbR^n$
we let $\mathbf{u} \approx \mathbf{v}$ if $u_i \approx v_i$ for all $i < n$, and
$\mathbf{u} \ge \mathbf{v}$ if $u_i \ge v_i$ for all $i < n$.
With this understanding, the material in Section~\ref{examples}, and in particular  Proposition~\ref{propW}, generalizes straightforwardly to functions $w$ with $\ran w \subseteq \bbR^n$.
One can then follow the proof of Theorem~\ref{Theorem}.
\end{proof}


\section{Applications of infinitesimal perturbations}
\label{perturbs}

Recall (see Conclusion of proof of Theorem~\ref{Theorem}) that
$y_{\bep} $ is a standard function defined via \eqref{starstar}.

\begin{lemma} \label{tags}
Let $F$ be standard. For every standard solution $y$ of \eqref{star}
defined on a standard interval $[0,a)$ and every standard $c < a$,
  $c>0$, there is an infinitesimal perturbation $\bep$ such that
  $y(x)= y_{\bep}(x)$ holds for $0 \le x \le c$.
\end{lemma}

\begin{proof}
By the mean value theorem, for each $k$ such that $x_{k+1} \le c$ there is  $t \in [x_k, x_{k+1}]$ such that 
$y(x_{k+1}) - y(x_k) = y'(t)\cdot h$.
Let $t_k$ be the least such $t$ (as $y'$ is continuous, the set of $t$ with this property is closed).
Then let $\ep_k = F(t_k, y(t_k)) - F(x_k, y(x_k)) = y'(t_k) -y'(x_k)\approx 0$. For $x_{k+1} > c$ let $\ep_k = 0.$
Let $y_0 = y(0)$; it follows  that $y_k = y(x_k) $ for all $k$ such that $x_{k+1} \le c$: 
assuming the claim is true for $k$, we have
$y_{k+1} =  y_k + (F(x_k, y_k) + \ep_k)\cdot h = y(x_k) +  F(t_k, y(t_k))\cdot h =y(x_k) + y'(t_k)\cdot h
= y(x_{k+1})$.

If $x \in [0, c]$ is standard, take $x \approx x_k$ for $x_{k+1} \le
c$; then $y_{\bep} (x) \approx y_k = y(x_k) \approx y(x)$, so
$y_{\bep} (x) = y(x)$.
\end{proof}

\begin{corollary}\label{extension}
Every solution of \eqref{star} extends to a global solution.
\end{corollary}

\begin{proof}
Let $y$ defined on $[0, c)$ be a standard solution of \eqref{star}
  with $F$ standard. If $y$ is not global, then it has a standard
  continuation $\tilde{y}$ to an interval $[0, a)$ with $c < a$. By
    Lemma~\ref{tags} $y$ has a continuation $y_{\bep}$ which is global
    by Theorem~\ref{Theorem}.  By Transfer, the claim holds for all
    solutions $y$ and all functions $F$.
\end{proof}

\begin{theorem}\label{perturb}
For every standard global solution $y$ of \eqref{star} there is an
infinitesimal perturbation $\bep$ such that $y= y_{\bep}$.
\end{theorem}

\begin{proof}
Assume the domain of $y$ is a standard interval $[0, a)$ (possibly $a =+\infty$).
We fix  a standard strictly increasing sequence $\la c_n \,\mid\, n \in \bbN\ra$ such that $c_0 > 0$ and $\lim_{n \to \infty} c_n = a$.
The proof of Lemma~\ref{tags} (with $c = c_n$) justifies the following statement:\\
For every standard $n \in \bbN$ there is $\bep = \la \ep_k \,\mid\, 0 \le k < N^2\ra$ such that 
for all $m \le n$ and for all  $ k< N^2$
$$ \left( x_{k +1} \le c_m \imp \ep_k =   y'(t_k) -  y'(x_k)\right) \,\wedge\, \left( x_{k +1} > c_m \imp |\ep_k| < \tfrac{1}{m+1}\right).
$$
By Countable Idealization  (Lemma~\ref{countideal}) there is $\bep$ such that for all standard $ n\in \bbN$ and for all  $ k< N^2$
$$ \left( x_{k +1} \le c_n \imp \ep_k =   y'(t_k) -  y'(x_k)\right) \,\wedge\, \left( x_{k +1} > c_n \imp |\ep_k| < \tfrac{1}{n+1}\right).
$$
It folows that $\ep_k \approx 0$ for all $k < N^2$, so $\bep$ is a perturbation.
As in the proof of Lemma~\ref{tags}, $y(x) = y_{\bep}(x)$ holds for every standard $x \in [0, c_n]$,  for every standard $n \in \bbN$, hence $y(x) = y_{\bep}(x)$ holds for every standard $x \in [0, a)$.
By Transfer, $y(x) = y_{\bep}(x)$  for all $x \in [0, a)$.
\end{proof}

The results of this section generalize to the system of equations
\eqref{star5}.


\section{Osgood's Theorem in {\bf SPOT}}
\label{osgoodsection}

\begin{definition}
A solution $\tilde{y}$ of \eqref{star} defined on an interval $I$ is
\emph{maximal on} $I$ if $\tilde{y}(x) \ge y(x)$ holds for every
solution $y$ of \eqref{star} and every $x \in I \cap \dom y$.  The
solution $\tilde{y}$ is \emph{maximal} if it is global and maximal on
its domain.
\end{definition}

\begin{theorem} \emph{(Global Osgood's Theorem)} \label{osgood}
The initial value problem \eqref{star} has a unique maximal solution.
\end{theorem}

\begin{proof}
We assume that $F$ is standard, fix an infinitesimal $\ep > 0$ and
consider the initial value problem
\begin{equation}
\tag{$\ast\!\ast\!\ast$}
\label{star3}
z(0)= 0,\quad
\quad z'(x) =F(x,z(x)) + \ep . 
\end{equation}

\begin{lemma} \label{L0} 
There exist standard $e, M> 0$ such that, for $ I= [0, e]$ and $J =
[-(M+1) \cdot e, (M+1) \cdot e]$, the function $F + \ep$ is bounded by
$M$ on $I \times J$ and the initial value problem \eqref{star3} has a
solution $u :I \to J$.
\end{lemma}

\begin{proof}[Proof of Lemma~$\ref{L0}$]
The arguments given in the proof of Theorem~\ref{Theorem} establish the following uniform result:

Given standard $c, d ,M >0$ there is a standard $e>0$ such that for every standard $G$, continuous and  bounded by $M$  on $[0,c) \times [-d, d]$, there is a solution 
$y: [0,e] \to [-(M+1)\cdot e, (M+1) \cdot e]$ of the initial value problem $y(0)=0$, $y'(x) = G(x,y(x))$.
By Transfer, the result holds for all such functions $G$. 

Returning to \eqref{star3}, fix standard $c,d,M_0 > 0$ so that $F$ is
bounded by $M_0$ on $[0,c) \times [-d, d]$.  Let $G = F + \ep$ and $M
  = M_0 + 1$. The paragraph above gives the desired solution $u$.
\end{proof}

\begin{lemma} \label{L1}
Let $u$ be the solution of the initial value problem \eqref{star3}
furnished by Lemma~\ref{L0} and let $y(0) = 0$ and $y'(x) = F(x,
y(x))$ for all $x \in [0, a)$. Then $u(x) \ge y(x)$ holds for all $x
  \in [0, \min\{e,a\})$.
\end{lemma}

\begin{proof}[Proof of Lemma~$\ref{L1}$]
Let 
\[
\alpha = \sup \{\alpha' \,\mid\, u(x) \ge y(x) \text{ holds for all }
0 \le x \le \alpha'\}
\]
and assume $\alpha < \min\{e,a\}$. 
Then $u(\alpha) \ge y(\alpha) $ and $u'(\alpha) - y'(\alpha) = \ep > 0$.
It follows that $u(x) >y(x)$ holds on some interval $(\alpha, \alpha']$ for $\alpha' > \alpha$, a contradiction.
\end{proof}

We next prove the existence of a local maximal solution.  We
let $$y_{m}(x) = {}^{\st} \{ \shh (u(x)) \,\mid\, x \in [0, e]\}.$$ The
existence of the standard function $y_m$ defined on $[0, e]$ in
$\SPOT$ follows from Proposition~\ref{propW} (with $I = [0,e]$ and $w
= u$), using the observation that $u$ is S-continuous: $|x - z
|\approx 0$ implies $ |u(x) - u(z)| = |u'(t)|\cdot|x -z|= |F(t, u(t))
+ \ep|\cdot|x -z| \le M \cdot | x - z| \approx 0$ (where $t$ is
between $x, z \in I$).

If $y $ is a standard solution of \eqref{star}, then $y_{m}(x) = \shh
(u(x)) \ge \shh y(x)) = y(x)$ holds for all standard $x \in [0,
  \min\{e,a\})$ by Lemma~\ref{L1}, so $y_{m}$ dominates all standard
  solutions of \eqref{star}.

\begin{lemma} \label{L2}
The function $y_{m}$ is a solution of \eqref{star} on $[0, e]$.
\end{lemma}

\begin{proof}[Proof of Lemma~$\ref{L2}$]
To this effect it suffices to find an infinitesimal perturbation
$\bep$ such that $y_m = y_{\bep}$ on $[0,e]$.

As in the proof of Theorem~\ref{tags}, for each $k$ with $x_{k+1} \le
e$ let $t_k$ be the least $t \in [x_k, x_{k+1}]$ such that $u(x_{k+1})
- u(x_k) = u'(t)\cdot h$.  Then let $\ep_k = F(t_k, u(t_k)) - F(x_k,
u(x_k))$; if $x_{k+1} > e$ let $\ep_k = 0.$

Let $y_0 = u(0) = 0$.
If  $y_k = u(x_k) $, then
\[
\begin{aligned}
y_{k+1} &= y_k + (F(x_k, y_k) + \ep_k)\cdot h \\&= u(x_k) + F(t_k,
u(t_k))\cdot h \\&= u(x_k) + u'(t_k)\cdot h =u(x_{k+1}).
\end{aligned}
\]
It follows that $y_k = u(x_k) $ for all $k$ such that $x_{k+1} \le e$. 

We still have to show that $\ep_k \approx 0$. 
The function $u$ is S-continuous: 
$x,z \in [0, e]$ and $x \approx z$ imply
\[
\begin{aligned}
|u(x) - u(z)| \le \Big|\int_{z}^{x} (F(t,u(t)) + \ep)\,dt \Big| \le
M\cdot |x - z| \approx 0.
\end{aligned}
\]
So $t_k \approx x_k$ implies $u(t_k) \approx u(x_k)$ and $ F(t_k,
u(t_k)) \approx F(x_k, u(x_k))$, because $F$ is continuous at $\la
\shh(x_k), \shh(u(x_k))\ra$.

For standard $x \in [0, e]$ take $x \approx x_k$; we have $y_{m}(x) =
\shh(u(x)) = \shh(u(x_k)) = \shh (y_k) = y_{\bep}(x)$.  By Transfer,
$y_m (x) = y_{\bep}(x)$ holds for all $x \in [0,e]$.
\end{proof}

The above  argument establishes the existence of a solution $y_m$ which is maximal over some interval $[0, e)$. The maximal solution $y_{\max}$ is obtained as the union of all such solutions; it is defined and maximal on some interval $I$. It remains to prove that $I =  [0,a)$ (with $0 < a \le +\infty$)  and that $y_{\max}$ is global.
  If $I =  [0,a]$ for $a \in \bbR$,
we could apply the above argument to the initial value $\la a, y_{\max}(a)\ra$ and obtain a  continuation of $y_{\max}$ that is defined and maximal on a larger interval. Similarly, if $y_{\max}$
could be continued to some (non-maximal) standard solution $y$, then we could apply the above argument to the initial value $\la a, y(a)\ra$.
\end{proof}

This concludes the proof of Theorem~\ref{osgood} for standard $F$. By
Transfer, the theorem is true for all $F$.  \qed

\section{Final Remarks.}

\begin{remark}\label{others}
The proofs of the global Peano theorem we found in the literature often simply appeal to Zorn's lemma (eg.~Ganesh~\cite{Ganesh}, Theorem~4.7).  The more careful proofs depend on $\ADC$, usually without  mentioning it explicitly.
 Hale~\cite{Hale} in his proof of global Peano theorem (Theorem 2.1, p. 17) writes:
\begin{enumerate}
\item[] ``{\ldots}~there is a monotone increasing sequence $\{b_n\}$
  constructed as above so that the solution $x(t)$ of (1.1) on $[a,
    b]$ has an extension to the interval $[a, b_n]$ and $(b_n,
  x(b_n))$ is not in $\bar{V_n}$. Since the $b_n$ are bounded above,
  let $\omega = \lim_{n \to \infty} b_n $. It is clear that $x $ has
  been extended to the interval $[a, \omega)$ {\ldots}''
\end{enumerate}
What is actually clear is that his construction yields  solutions $x_n(t)$ on $[a, b_n]$  for each $n$, and each  $x_n(t)$ has extensions to some $x_{n+1}(t)$.
The axiom $\ADC$ is needed to justify the existence of $x(t)$.
Similarly Hartman~\cite[ II,~3.1, p. 13] {Ha} constructs an increasing sequence $\{ b_n \}$  such that any solution on $[a, b_n]$ has an extension to a solution on $[a, b_{n+1}]$. 
Here $\ADC$ is needed to justify the existence of a solution on $[a, \omega_{+}]$ for $\omega_{+} = \lim_{n  \to \infty} b_n$.   In Hartman's proof of III, Lemma 2.1, a key step to the proof of III, Theorem 2.1 (Osgood's theorem),  $\ACC$ is used implicitly to obtain the sequence $\{ u_n(t) \}$.
Similar unacknowledged use of $\ADC$ appears in  Kurzweil~\cite{Ku86}, pp.\;355--356.

\end{remark}

\begin{remark}\label{reverse}
Simpson~\cite{Si2} carried out a thorough study of the axioms needed to prove the \emph{local} versions of Peano and Osgood theorems. He showed that (over $\mathbf{RCA}_0$) the local Peano theorem is equivalent to $\mathbf{WKL}_0$ and the local Osgood theorem is equivalent to $\mathbf{ACA}_0$ (see  Simpson~\cite{Si} for the description of these systems of second order arithmetic and additional information). In particular, the proofs of local versions of these theorems do not need any form of $\AC$.

\end{remark}

\begin{remark}\label{absoluteness}
The conservativity of $\SPOT$ over $\ZF$ and the results of this paper imply that global Peano and Osgood theorems are provable in $\ZF$.

In a discussion on MathOverflow~\cite{MO}, James Hanson pointed out
that the same conclusion follows from Shoenfield's absoluteness
theorem.  A consequence of this theorem is that every $\Pi^1_4$
sentence provable in $\ZFC$ is provable in $\ZF$ alone.  The global
Peano theorem can be expressed by a $\Pi^1_4$ sentence, and therefore
it is provable in $\ZF$. The $\ZF$ proof obtained by conversion of the
$\ZFC$ proof by this method is far from elementary; in addition to
Shoenfield's absoluteness theorem, it relies on the notion of
relatively constructible sets.

\end{remark}

\textbf{Clarification of a point in~\cite{HK}.}  In Section~4
of~\cite{HK}, $\cM$-\emph{generic filters} on a forcing notion $\bbP
\in \cM$ are defined (see Definition 4.10).  Following a paragraph
that explains how such filters are constructed, it is stated that
``$\cM$-generic filters $\cG \subseteq M \times M$ on $\bbH$ are
defined and constructed analogously.''  There is a difference though,
in that the forcing notion $\bbH$ is a proper class from the point of
view of $\cM$.  The $\cM$-generic filters $\cG \subseteq M \times M$
on $\bbH$ have to meet every class $D \subseteq M$ which is definable
in $\cM$ (with parameters from $M$) and dense in $\bbH$.  As there are
only countably many such classes, the construction of a generic filter
on $\bbH$ can proceed analogously to the construction of a generic
filter on $\bbP$.

\section*{Acknowledgments}

We are grateful to Dalibor Pra\v{z}\'{a}k for helpful comments.

\bibliographystyle{amsalpha}
\end{document}